\documentclass{amsart}
\usepackage{verbatim}
\usepackage{amssymb}
\usepackage{amsbsy}
\usepackage{amscd}
\usepackage{amsmath}
\usepackage{amsthm}
\usepackage{amsxtra}
\usepackage{latexsym}
\usepackage[mathscr]{eucal}
\newcommand{\Ass}[1]{X_{#1}}
\newcommand{\Ring}[1]{R_{#1}}
\newcommand{\VA}{V}

\newcommand{\mc}{\mathcal}
\newcommand{\mf}{\mathfrak}

\newcommand{\on}{\operatorname}

\newcommand{\Vir}{Vir}

\newcommand{\isomap}{{\;\stackrel{_\sim}{\to}\;}}

\newcommand{\nc}{\newcommand}
\nc{\Hp}[1]{H^{#1}}

\newcommand{\affg}{\widehat{\mathfrak{g}}}

\newcommand{\fing}{\mathfrak{g}}

\newcommand{\p}{\mf{p}}

\newcommand{\N}{\mathbb{N}}
\newcommand{\Q}{\mathbb{Q}}
\newcommand{\1}{{\mathbf{1}}}

\newcommand{\bra}{{\langle}}
\newcommand{\ket}{{\rangle}}

\newcommand{\nno}{\nonumber}
\newcommand{\Lam}{\Lambda}

\newcommand{\lam}{\lambda}
\newcommand{\ra}{\rightarrow}
\newcommand{\+}{\mathop{\oplus}}
\newcommand{\Z}{\mathbb{Z}}

\newcommand{\inv}{^{-1}}

\renewcommand{\*}{{\otimes}}
\newcommand{\C}{\mathbb{C}}
\newcommand{\che}{^{\vee}}

\theoremstyle{plain}
\newtheorem{Th}{Theorem}[subsection]

\newtheorem{Pro}[Th]{Proposition}

\newtheorem{Lem}[Th]{Lemma}

\newtheorem{Co}[Th]{Corollary}

\theoremstyle{definition}

\theoremstyle{remark}

\newtheorem{Rem}[Th]{Remark}

\DeclareMathOperator{\Ann}{Ann}

\DeclareMathOperator{\id}{id}

\DeclareMathOperator{\End}{End}
\DeclareMathOperator{\gr}{gr}
\DeclareMathOperator{\Hom}{Hom}

\DeclareMathOperator{\Ad}{Ad}
\DeclareMathOperator{\haru}{span}

\DeclareMathOperator{\Spec}{Spec}

\title{A Remark on the $C_2$-cofiniteness condition  on vertex algebras}
\author{Tomoyuki Arakawa}

\address{Department of Mathematics, 
Nara Women's University, Nara 630-8506, JAPAN}

\curraddr{Research Institute for Mathematical Sciences, Kyoto University,
Kyoto 606-8502,
JAPAN}

\email{arakawa@kurims.kyoto-u.ac.jp}

\thanks{This work  is partially  supported 
by the JSPS Grant-in-Aid  for Scientific Research (B)
No.\ 20340007}



\begin{document}
\maketitle

\begin{abstract}
We show that 
a finitely strongly generated,
non-negatively graded vertex algebra $V$ 
 is $C_2$-cofinite if and only if 
it is lisse in the sense of 
Beilinson, Feigin and Mazur
\cite{BeiFeiMaz}. 
This shows that 
the $C_2$-cofiniteness  is indeed a natural finiteness
condition. 

\end{abstract}

\section{Introduction}
The purpose of this note is to 
clarify 
the equivalence of the two finiteness conditions
on vertex algebras.

One of them
is the  finiteness condition 
introduced by Zhu \cite{Zhu96},
which is now called the
{\em $C_2$-cofiniteness condition}.
This condition
 has 
been assumed
in many
significant  theories of vertex operator algebras,
such as \cite{Zhu96,DonLiMas00,Miy04,NagTsu05,Hua08,Hua08rigidity}.
However
the original definition 
looks rather technical,
 and hence
it
  has been often considered  
as
  a mere technical condition.

The other is
the  finiteness
condition
 defined by
Beilinson, Feigin and Mazur \cite{BeiFeiMaz}
in the case of the Virasoro (vertex) algebra,
which is called {\em lisse}.
The definition is  as follows.
Let $\mc{L}$ be the Virasoro (Lie) algebra,
$U(\mc{L})$ its universal enveloping algebra.
There is a natural increasing filtration of $U(\mc{L})$
in the Lie theory,
called the {\em standard filtration},
and the associated graded  algebra
$\gr U(\mc{L})$
of $U(\mc{L})$
is isomorphic to the symmetric algebra
$S(\mc{L})$ of $\mc{L}$.
Let $M$ be a finitely generated $\mc{L}$-module,
$\{ \Gamma_p M\}$ a good filtration
(i.e., a filtration compatible with the standard
filtration of $\mc{L}$ such  that
the associated graded space $\gr^{\Gamma}M$
is finitely generated over $S(\mc{L})$).
The  {\em singular support} $SS(M)$ of $M$
is 
the support
of the $S(\mc{L})$-module $\gr^{\Gamma}M$,
which is known to be independent of the choice of a good filtration.
A $\mc{L}$-module $M$ is called lisse if
$\dim SS(M)=0$.
In the case that  $M$ is a highest weight representation
of $\mc{L}$
then
$M$ is lisse
if and only if
any element of $\mc{L}$ acts locally nilpotently 
on $\gr^{\Gamma } M$.
This implies that
 lisse representations are natural analogues
of finite-dimensional representations.

Let
$V$ be
a 
  finitely strongly generated,
non-negatively graded vertex algebra.
The notion of singular supports  
 can be naturally extended to $V$
by using the canonical 
filtration introduced by
 Li \cite{Li05}.
More precisely,
 define 
 the {\em associated variety}
$X_M$ of a $V$-module $M$
as the support of the $V/C_2(V)$-module $M/C_2(M)$.
Then
 the singular support of $M$ can be 
naturally  defined as a subscheme 
of the infinite jet scheme of
$X_M$.
Having defined the associated variety and the singular support,
 the equivalence 
of the $C_2$-cofiniteness condition 
and the lisse condition
easily follows from the fact that
the jet scheme of a zero-dimensional variety is zero-dimensional
(see Theorems \ref{Th:C2-vs-lisse}, \ref{Th:C2-vs-lisse-modules}).

We note that in the case that $V$ is a Virasoro vertex algebra
and $M$ is a $V$-module,
the Li filtration of  $M$
is slightly different  from 
the Lie theoretic
filtration defined by considering it as a module over the Virasoro 
(Lie) algebra
$\mc{L}$.
Nevertheless 
$M$ is lisse if and only if it is so in the sense of \cite{BeiFeiMaz}
provided that
$M$ is 
 a highest weight representation
of $\mc{L}$
(see Proposition \ref{Pro:comparison}).
In the case that $V$ is an affine vertex algebra
associated with a Lie algebra
the Li filtration is essentially the same as  
the Lie theoretic filtration,
although it is a {\em decreasing} filtration
(see Propositions \ref{Pro:F=G}, 
\ref{Pro:in-the-case-of-affine-vertex-algebra}).

\smallskip

In a separate  paper \cite{Ara09b}
we 
 prove 
the $C_2$-cofiniteness  for  a large family
of $W$-algebras
using the methods in this note,
including
all 
the (non-principal)
exceptional $W$-algebras 
recently discovered by
Kac and Wakimoto \cite{KacWak08}.

\medskip

\subsection*{Notation}
The ground field
will be $\C$
throughout the note .
\section{Graded Poisson vertex algebras associated with vertex algebras}
\label{section:associated-graded}
\label{section:VPA}
\subsection{Vertex algebras and their modules}
A (quantum) {\em field} on a vector space $\VA$ is a formal series
\begin{align*}
 a(z)=\sum_{n\in \Z}a_{(n)}z^{-n-1}\in (\End \VA)[[z,z\inv]]
\end{align*}
such that $a_{(n)}v=0$ with $n\gg 0$ for any $v\in \VA$.
Let $Fields(\VA)$ denote the space of all fields on $\VA$.

A {\em vertex algebra} is a vector space $\VA$ equipped 
with the following data:
\begin{itemize}
 \item a linear map $Y(?,z): \VA\ra Fields(\VA)$,
$a\mapsto a(z)=\sum_{n\in \Z}a_{(n)}z^{-n-1}$,
\item a vector $\1\in \VA$, called the {\em vacuum vector},
\item a linear operator $T: \VA\ra \VA$,
called the {\em translation operator}.
\end{itemize}
These data are subjected to satisfy
 the following axioms:
\begin{enumerate}
 \item $(Ta)(z)=\partial_z a(z)$,
\item $\1(z)=\id_\VA$,
\item $a(z)\1\in \VA[[z]]$ and 
$a_{(-1)}\1=a$,
\item
\begin{align*}
& \sum_{j\geq 0}\begin{pmatrix}
		m\\j
	       \end{pmatrix}(a_{(n+j)}b)_{(m+k-j)}\\
&=
\sum_{j\geq 0}(-1)^j \begin{pmatrix}
		      n\\ j
		     \end{pmatrix}
\left(a_{(m+n-j)}b_{(k+j)}-(-1)^n b_{(n+k-j)}a_{(m+j)}\right).
\end{align*}
\end{enumerate}

A {\em Hamiltonian} of $\VA$
is a semisimple operator $H$ on $V$
satisfying
\begin{align}
 [H,a_{(n)}]=-(n+1)a_{(n)}+(Ha)_{(n)}
\label{eq:Hamiltonian}
\end{align}
for all $a\in V$,
$n\in \Z$.
A vertex algebra equipped with a Hamiltonian $H$
is called {\em graded}.
Let $V_{\Delta}=\{a\in V; H a=\Delta a\}$ for $\Delta\in \C$,
so that $V=\bigoplus_{\Delta\in \C}\VA_{\Delta}$.
For $a\in \VA_{\Delta}$,
 $\Delta$ is called the {\em conformal weight} of $a$
and denoted by $\Delta_a$.
We have
\begin{align*}
 a_{(n)}b\in \VA_{\Delta_a+\Delta_b-n-1}
\end{align*}
for homogeneous elements $a,b\in \VA$.

{\em Throughout the note 
$V$ is assumed to be $\frac{1}{r_0}\Z_{\geq 0}$-graded
for some $r_0\in \N$},
that is,
$V$ is graded and  $V_{\Delta}=0$ for $\Delta\not\in \frac{1}{r_0}\Z_{\geq 0}$.

A graded vertex algebra $\VA$ is
called {\em conformal} if there exists a vector $\omega$
(called the {\em conformal vector}) 
and $c_V\in \C$ (called the {\em central charge})
satisfying
\begin{align*}
& \omega_{(0)}=T,\quad 
\omega_{(1)}=H,\quad \\
&[\omega_{(m+1)},\omega_{(n+1)}]=(m-n)\omega_{(m-n+1)}
+\frac{(m^3-m)}{12}\delta_{m+n,0}c_V.
\end{align*}
A $\Z$-graded conformal vertex algebra is also called a
{\em vertex operator algebra}.

A vertex algebra $\VA$ is called {\em finitely strongly generated}
if there exists a finitely many elements $a^1,\dots,a^r$
such that
$V$ is spanned by  the elements of the form
\begin{align}
 a^{i_1}_{(-n_1)}\dots a^{i_r}_{(-n_r)}\1
\end{align}
with $r\geq 0$,
$n_i\geq 1$.

A {\em module} over a vertex algebra $\VA$ is a vector space $M$
together with a linear map
\begin{align*}
 Y^M(?,z): \VA\ra Fields(M),\quad
a\mapsto a^M(z)=\sum_{n\in \Z}a_{(n)}^M z^{-n-1},
\end{align*}
which satisfies the following axioms:
\begin{enumerate}
 \item $\1^M(z)=\id_M$,
\item
\begin{align*}
& \sum_{j\geq 0}\begin{pmatrix}
		m\\j
	       \end{pmatrix}(a_{(n+j)}b)_{(m+k-j)}^M\\
&=
\sum_{j\geq 0}(-1)^j \begin{pmatrix}
		      n\\ j
		     \end{pmatrix}
\left(a_{(m+n-j)}^M b_{(k+j)}^M-(-1)^n b_{(n+k-j)}^Ma_{(m+j)}^M\right).
\end{align*}
\end{enumerate}

A $V$-module $M$  is called
{\em graded} 
if there is a compatible semisimple action of $H$ on $M$,
that is,
$M=\bigoplus_{d\in \C}M_d$,
where
$M_d=\{m\in M; Hm=dm\}$,
and 
$
[H,a^M_{(n)}]=-(n+1)a_{(n)}^M+(H a)^M_{(n)}
$
for all $a\in V$.
We have
 $a_{(n)}^M M_d\subset M_{d+\Delta_a-n-1}$ for 
a homogeneous element $a\in \VA$.
If $\VA$ is conformal and $M$ is a $\VA$-module
on which  $\omega^M_{(1)}$ acts locally finite,
then $M$ is graded by the semisimplification of
$\omega_{(1)}^M$.

If  no confusion can arise
 we write simply $a_{(n)}$ for $a_{(n)}^M$.

{\em Throughout the note a $V$-module $M$
is assumed to be lower truncated},
that is, 
$M$ is graded and
there exists a finites subset
$\{d_1,\dots,d_r\}\subset \C$
such that
$M_d\ne 0$ unless $d\in d_i+\frac{1}{r_0}\Z_{\geq 0}$
for some $i$.

For a $V$-module $M$,
set
\begin{align*}
& C_2(M)=\haru_{\C}\{a_{(-2)}m; a\in \VA, m\in M\}.
\end{align*}
Then 
\begin{align*}
C_2(M)=\haru_{\C}\{a_{(-n)}m; a\in \VA, m\in M,\ n\geq 2\}
\end{align*}
by the axiom (i) of vertex algebras.
A $\VA$-module $M$ is called 
{\em $C_2$-cofinite}
if $\dim M/C_2(M)<\infty$,
and $V$ is called $C_2$-cofinite
if it is $C_2$-cofinite as a module over itself.

Suppose  that $V$ is conformal.
Then
a  $V$-module $M$ is called
{\em $C_1$-cofinite} \cite{Li99}
 if $M/C_1(M)$
is finite-dimensional,
where
\begin{align*}
&C_1(M)=\haru_{\C}\{\omega_{(0)}m, a_{(-1)}m;a\in 
\bigoplus_{\Delta>0}V_{\Delta},
m\in M
\}.
\end{align*}
If 
$V_0=\C \1$,
then $C_2(V)$ is a subspace of $C_1(V)$
(\cite[Remark 3.2]{Li99}).

\subsection{Vertex Poisson algebras and their modules}
A vertex algebra $\VA$ is said to be {\em commutative}
if $a_{(n)}=0$ in $\End \VA$ for all $n\geq 0$, $a\in \VA$.
This is equivalent to that
\begin{align*}
 [a_{(m)}, b_{(n)}]=0\quad \text{for }a,b\in V,\ m,n\in \Z.\end{align*}
A commutative vertex algebra
is the same as a differential algebra
(=a unital commutative algebra with a derivation) \cite{Bor86}:
the multiplication is given by
$a\cdot b=a_{(-1)}b$ for $a,b\in \VA$;
The derivation is given by the translation operator $T$.

A commutative vertex algebra $\VA$
is called a {\em vertex Poisson algebra} 
\cite{FreBen04}
if it is also  equipped with a linear operation 
\begin{align*}
Y_-(?,z)
:\VA\ra \Hom (\VA,z\inv \VA[z\inv]),
\quad a\mapsto a_-(z),
\end{align*}
such that
\begin{align}
 &(Ta)_{(n)}=-n a_{(n-1)},\label{eq:axiom1-for-VPA}\\
&a_{(n)}b=\sum_{j\geq 0}(-1)^{n+j+1}\frac{1}{j!}T^j (b_{(n+j)}a),
\label{eq:axiom2-for-VPA}\\
&[a_{(m)},b_{(n)}]=\sum_{j\geq 0}\begin{pmatrix}
					m\\j
				       \end{pmatrix}(a_{(j)}b)_{(m+n-j)},
\label{eq:axiom3-for-VPA}
\\
&a_{(n)}(b \cdot  c)=(a_{(n)}b)\cdot c+b\cdot (a_{(n)}c)
\label{eq:axiom4-for-VPA}
\end{align}
for $a,b,c \in \VA$ and $n,m\geq 0$.
Here, by abuse of notation,
we have set
\begin{align*}
 a_-(z)=\sum_{n\geq 0}a_{(n)}z^{-n-1}.
\end{align*}

Below, for an element $a$
of a vertex Poisson algebra,
we will denote by
 $a_{(n)}$ with $n\leq -1$ the  Fourier 
coefficient of $z^{-n-1}$ in  the field $Y(a,z)=a(z)$,
and by 
   $a_{(n)}$ with $n\geq 0$ the Fourier 
coefficient of 
$z^{-n-1}$ in  $Y_-(a,z)=a_-(z)$.

A vertex Poisson algebra $V$ is called {\em graded} if there exists a
semisimple operator $H$ on $V$
satisfying 
\eqref{eq:Hamiltonian}
for all $n\in \Z$.

A {\em module} over a vertex Poisson algebra 
$\VA$ is a module  $M$ over 
$\VA$ as an associative algebra
together with a linear  map
\begin{align*}
 Y^M_-(?,z): \VA\ra \Hom(M,z\inv M[z\inv]),\quad
a\mapsto a_-^M(z)=\sum_{n\geq 0}a_{(n)}^M z^{-n-1},
\end{align*}
such that 
\begin{align}
 &\1^M(z)=\id_M,\nno\\
&(Ta)_{(n)}^M=-n a_{(n-1)}^M, \nno\\
&[a_{(m)}^M,b_{(n)}^M]=\sum_{i\geq 0}\begin{pmatrix}
					m\\ i
				       \end{pmatrix}
(a_{(i)}b)_{(m+n-i)}^M, \nno\\
&a_{(n)}^M(b\cdot m)=(a_{(n)}b)\cdot m+b\cdot (a_{(n)}^Mm).
\label{eq:module-for-vertex-Poisson}
\end{align}
for $a,b\in \VA$, $m\in M$ and $n\geq 0$.

A module $M$ over a graded vertex Poisson algebra $V$
is called {\em graded} if there exists a semisimple
operator $H^M$ on $M$
satisfying \begin{align}
 [H^M,a^M_{(n)}]=-(n+1)a^M_{(n)}+(Ha)^M_{(n)}
\label{eq:Hamiltonian2}
\end{align}
for all $a\in V$,
$n\in \Z$.

If  no confusion can arise
 below we write  $a_{(n)}$ for $a_{(n)}^M$
and $H$ for $H^M$.

We   call  $M$ a 
 {\em differential $\VA$-module}  
if
it is
equipped with the linear action of $T$
satisfying
\begin{align*}
[T,a_{(n)}]=-n a_{(n-1)}\quad \text{for all }n\in \Z.
\end{align*}
\subsection{Functions on jet schemes of affine Poisson varieties
are vertex Poisson algebras}
For a scheme $X$ of finite type 
we denote by
 $X_m$ be the jet scheme of order $m$
of $X$
and by $X_{\infty}$
the infinite jet scheme
of $X$.

Let us recall the definition of jet schemes.
For general theory of jet schemes see e.g., \cite{EinMus}.
The connection between jet schemes and chiral algebras 
goes back to Beilinson and Drinfeld \cite{BeiDri04}.
The scheme $X_m$ is  determined by its functor of points:
for every commutative  $\C$-algebra $A$,
there is  a bijection
\begin{align*}
 \Hom(\Spec A, X_m)\cong \Hom(\Spec A[t]/(t^{m+1}),X).
\end{align*}
If $m>n$, we have projections
$ X_m\ra X_n$.
This yields  a projective system
$\{X_m\}$ of schemes,
and the infinite jet scheme $X_{\infty}$
is the projective limit
$\lim\limits_{\underset{m}{\leftarrow}}
X_{m}$
in the category of schemes.

For an affine scheme $X=\Spec R$,
the jet scheme $X_m$ is explicitly described.
Choose a presentation  
$R=\C[x^1,\dots, x^r]/\bra f_1,\dots,f_s\ket$.
Define new variables $x^j_{(-i)}$ for $i=1,\dots,m+1$
and a derivation $T$ of the
 ring $\C[x^j_{(-i)};i=1,2,\dots,m+1,\ 
 j=1,\dots,r
]$
 by setting
\begin{align*}
 T x^j_{(-i)}=\begin{cases}
	    i x^j_{(-i-1)}&\text{for $i\leq m$},\\
0&\text{for $i=m+1$}.
	   \end{cases}
\end{align*}
Identify $x^j$ with $x^j_{(-1)}$
and set
\begin{align*}
 R_m=\C[x^j_{(-i)}; i=1,\dots,m+1,\ j=1,\dots,r
]
/
\bra T^j f_i;
i=1,\dots, s,\
j=0,\dots,m+1\ket,
\end{align*}
and we have
$X_m \cong \Spec R_m$.

Let $R_{\infty}$ denotes the differential algebra
obtained from $R_m$ by taking the limit $m\ra {\infty}$:
\begin{align}
 R_{\infty}=\C[x^j_{(-i)}; i=1,2,\dots,\
j=1,\dots,
]/
\bra T^j f_i;i=1,\dots, s,
j=0,\dots, \ket.
\label{eq:R-infty}
\end{align}
Then 
$
X_{\infty}\cong \Spec (R_{\infty})
$.

\begin{Pro}\label{Pro:R-infty-is-v-poisson}
 Let $R$ be a Poisson algebra. 
Then there is a unique vertex Poisson algebra structure
on
 $R_{\infty}$ 
such that
\begin{align*}
 u_{(n)}v=\begin{cases}
	   \{u,v\}&\text{if }n=0,\\
0&\text{if }n>0,
	  \end{cases}
\end{align*}
for $u,v\in R\subset R_{\infty}$.
\end{Pro}
\begin{proof}
Set
\begin{align}
u_{(n)}(T^l v)=\begin{cases}
		 \frac{l!}{(l-n)!}T^{l-n}\{u,v\}&\text{if }l\geq n,\\
0&\text{else}
		\end{cases}
\label{eq:VPA-structure0}
\end{align}
for $a,b\in R$.
This
extends to  a well-defined linear map
\begin{align}
R\ra \on{Der}(R_{\infty})[[z\inv]]z\inv,\quad
u\mapsto u_-(z)=\sum_{n\in \geq 0}u_{(n)}z^{-n-1},
\label{eq:1st}
\end{align}
where $\on{Der}(R_{\infty})$ is  the space of 
derivations on $R_{\infty}$.
It is straightforward to check that
\begin{align}
&[T,u_-(z)]=\partial_z u_-(z)\quad \text{for }u\in R,\nno \\
&u_{(n)}v=\sum_{j\geq 0}(-1)^{n+j+1}\frac{1}{j!}T^j(v_{(n+j)}u)
\quad\text{for }u,v\in R,\nno \\
&[u_{(m)},v_{(n)}]=\sum_{j\geq 0}\begin{pmatrix}
				  m\\j
				 \end{pmatrix}(u_{(j)}v)_{(m+n-j)}
\quad\text{for }u,v\in R.\label{eq:commutattor-ok}
\end{align}
 
The map (\ref{eq:1st})
 can be extended to the linear map
\begin{align}\label{eq:VPA-structure-on-the-way}
R_{\infty}\ra \on{Der}(R_{\infty})[[z\inv]]z\inv,\quad
a\mapsto a_-(z)=\sum_{n\geq 0}a_{(n)}z^{-n-1}
\end{align}
by setting
\begin{align*}
 a_-(z)T^ku=\on{Sing}(e^{zT}(-\partial_z)^k u_-(-z)a)
\end{align*}
for $a\in R_{\infty}$,
$u\in R$, $k\geq 0$.
Here $\on{Sing}(f)$ is the singular part of the formal series $f$
(\cite{Pri99,Li04}).
(Note that the axiom (\ref{eq:axiom2-for-VPA})
is equivalent to that 
$a_{-}(z)b=\on{Sing}(e^{z T}b_-(-z)a)$.)
We find from
the argument of \cite[Proof of Proposition 3.10]{Li04}
that
\begin{align*}
& (Ta)_-(z)=\partial_z a_-(z)\quad \text{for }a\in R_{\infty},\\
&a_-(z)b=\on{Sing}(e^{zT}b_-(-z)a)\quad\text{for }a,b\in R_{\infty}.
\end{align*}
This together
(\ref{eq:commutattor-ok}) proves
that (\ref{eq:VPA-structure-on-the-way})
defines a vertex Poisson structure on $R_{\infty}$
by \cite[Theorem 3.6]{Li04}.
The uniqueness statement is easily seen,
cf.\  \cite[Proof of Proposition 3.10]{Li04}.
\end{proof}

The vertex Poisson algebra structure 
on $R_{\infty}=\C[X_{\infty}]$
given in Proposition \ref{Pro:R-infty-is-v-poisson}
for an affine scheme $X=\Spec R$
will be  called the {\em level $0$ vertex Poisson algebra
structure}.
\begin{Rem}
The differential 
algebra  $R_{\infty}$ makes sense even when 
$R$ is not finitely generated and
 Proposition \ref{Pro:R-infty-is-v-poisson}
still holds in this case.
Also,
the assertion of Proposition \ref{Pro:R-infty-is-v-poisson}
extends straightforwardly to the Poisson superalgebra cases.
\end{Rem}
\subsection{Ideals of vertex Poisson algebras}
Let $\VA$ 
be a vertex Poisson algebra,
$M$ a $V$-module.
A {\em submodule}  of $M$
is a submodule $N$ of $M$ as an associative algebra
such that $a_{(n)}N\subset N$ for all $a\in \VA$,
$n\geq 0$.
An {\em ideal} of $\VA$ is a submodule $I$ of $\VA$ 
stable under the action of $T$.
By (\ref{eq:axiom2-for-VPA})
it follows that
$u_{(n)}\VA\subset I$ for all $u\in I$
and $n\geq 0$.
Hence 
$\VA/I$  inherits the vertex Poisson algebra
from $V$.

\begin{Lem}
 Let $R$ be a commutative $\Q$-algebra,
$\partial$ a derivation on $R$,
$I$ a $\partial$-stable ideal of $R$.
Then  the radical $\sqrt{I}$ of $I$ is  stable under the action 
of $\partial$.
\end{Lem}
\begin{proof}(\cite{CasOsb78})
Let $a\in \sqrt{I}$,
so that there exists $m\in \N$
such that $a^m\in I$.
Because $\partial I\subset I$,
$ (\partial)^m a^m \in I$.
But \begin{align*}
      (\partial)^m a^m\equiv m! (\partial a)^m\pmod{\sqrt{I}},
    \end{align*}
and therefore
we get that $(\partial a)^m\in \sqrt{I}$.
This gives  $\partial a\in \sqrt{I}$.
\end{proof}

\begin{Co}\label{Co:radical-of-VPA-ideal}
 \begin{enumerate}
\item Let $R$ be a Poisson algebra,
$I$ a Poisson ideal of $R$.
Then  $\sqrt{I}$ is also a Poisson ideal.
\item Let $\VA$ be a vertex Poisson algebra,
$I$ a $\VA$-submodule of $\VA$.
Then $\sqrt{I}$ is also a $\VA$-submodule of $\VA$.
\item Let $\VA$ be a vertex Poisson algebra,
$I$ a vertex Poisson algebra ideal of $\VA$.
Then  $\sqrt{I}$ is also a vertex Poisson algebra ideal of $\VA$.
 \end{enumerate}
\end{Co}

\subsection{Li filtration}
Let $\VA$ be a  vertex algebra.
Following \cite{Li05}\footnote{In \cite{Li05}
$F^p V$ was denoted by $E_p$.},
define  $F^p\VA$ to be the
subspace of $\VA$
spanned by the vectors
\begin{align*}
 a^1_{(-n_1-1)}\dots a^r_{(-n_r-1)}b
\end{align*}
with $a^i\in \VA$,
$b\in \VA$
$n_i\in \Z_{\geq 0}$,
$n_1+\dots+n_r\geq p$.
Then 
\begin{align}
& \VA=F^0\VA\supset F^1\VA\supset \cdots,
\quad
\bigcap_p F^p V=0,\nno
\\&
T F^p\VA\subset F^{p+1}\VA, \nno\\
&
 a_{(n)}F^q \VA\subset F^{p+q-n-1}\VA\quad \text{for } a\in F^p \VA,\  n\in \Z,
\nno \\
&
 a_{(n)}F^q\VA\subset  F^{p+q-n}\VA\quad \text{for } 
 a\in
 F^p \VA,\ n\geq 0.
\label{eq:Li-fitration2}
\end{align}
Here we have set $F^p V=V $ for $p<0$.
Note that the filtration
$\{F^pV\}$ is independent of the grading of $V$.

Let $\gr^F\VA=\bigoplus_{p}F^p\VA/F^{p+1}\VA$
be the associated graded vector space.
The space $\gr^F \VA$ is  a vertex Poisson algebra
by
\begin{align*}
 &\sigma_p(a)\sigma_q(b)=\sigma_{p+q}(a_{(-1)}b)
\\
&T\sigma_p(a)=\sigma_{p+1}(Ta),\\
&
 Y_-(\sigma_p(a),z)\sigma_q(b)=\sum_{n\geq 0}\sigma_{p+q-n}(a_{(n)}b)z^{-n-1},
\end{align*}
where 
$\sigma_p: F^p \VA\ra F^p\VA/F^{p+1}\VA$ is the principal symbol map.

The filtration  $\{F^p \VA\}$ is called  the {\em Li filtration}
of $\VA$.


We have  \cite[Lemma 2.9]{Li05}
\begin{align}
 F^p \VA=\haru_{\C}\{a_{(-i-1)}b;
a\in \VA,
i\geq 1, b\in F^{p-i}\VA\}\quad \text{for all $p\geq 1$.}
\label{eq:2009:04:02:10:39}
\end{align}
In particular
\begin{align*}
F^1 M=C_2(M).
\end{align*}
Set
\begin{align*}
 \Ring{\VA}=
\VA/C_2(\VA)=F^0\VA/F^1 \VA\subset \gr^F V.
\end{align*}
It is known by
Zhu \cite{Zhu96} 
that $\Ring{\VA}$ is a Poisson algebra.
In fact
the 
Poisson algebra structure
of $\Ring{\VA}$
can be understood as the
restriction of 
the vertex Poisson structure of $\gr^F V$
\cite[Proposition 3.5]{Li05};
it is given by
\begin{align*}
& \bar a\cdot \bar b=\overline{a_{(-1)}b},\\
&\{\bar a,\bar b\}=\overline{a_{(0)}b}
\end{align*}
for $a,b\in V$,
where $\bar a=a+C_2(V)$.

By \cite[Lemma 4.2]{Li05},
 the embedding $\Ring{\VA}\hookrightarrow \gr^F \VA$ 
extends to the surjective homomorphism
\begin{align}
(\Ring{\VA})_{\infty}\twoheadrightarrow \gr^F \VA
\label{eq:grV-is-a-quotient}
\end{align}
of differential algebras.

\begin{Pro}\label{Pro:maps-from-jet-scheme}
The surjection (\ref{eq:grV-is-a-quotient})
is a homomorphism of vertex Poisson algebras,
where $(\Ring{\VA})_{\infty}$ is equipped with the level $0$
vertex Poisson algebra structure.
\end{Pro}
\begin{proof}
From the definition 
we see that the vertex Poisson algebra structure
coincides  
on the
generating subspace $\Ring{\VA}$,
that is,
for $a,b\in \Ring{\VA}$,
$a_{(0)}b=\{a,b\} $
and $a_{(n)}b=0$ for all $n>0$
in both $\gr V$
and $(\Ring{\VA})_{\infty}$.
But  then  \cite[Lemma 3.3]{Li04}
says 
that
the map $(\Ring{\VA})_{\infty}\ra \gr V$ 
must be  a vertex Poisson algebra homomorphism.
\end{proof}

\subsection{Standard filtration vs Li filtration}
There is 
an another 
 filtration $\{G_p \VA; p\in \frac{1}{r_0}\Z_{\geq
0}\}$
of $\VA$ called the {\em standard filtration},
which is an {\em increasing} filtration
defined also  by Li \cite{Li04}\footnote{In \cite{Li04}
it is assumed that $V$ is  $\Z_{\geq 0}$-graded and
$V_0=\C$,
but this condition can be easily relaxed,
as we can see
 from 
Proposition \ref{Pro:F=G}.}:
choose 
 a set $\{a^i ; i\in I\}$ of  homogeneous  strong generators 
of $V$.
Let 
 $G_p \VA$ 
be the linear subspace of $\VA$
spanned by the vectors 
\begin{align}
 a^{i_1}_{(-n_1)}\dots a^{i_r}_{(-n_r)}\1
\quad \text{satisfying}\quad 
\Delta_{a^{i_1}}+\dots +\Delta_{a^{i_r}}\leq p,
\label{eq:2009-04-26-04-26}
\end{align}
with $r\geq 0$, $n_i\geq 1$.
Then
\begin{align}
& G_p\VA\subset G_{q}\VA\text{ for }p<q,
\label{eq:standard-filtration1}\\
&
\VA=\bigcup_{p} G_p \VA,
\\
&
T G_p \VA \subset G_p \VA,
\\
&a_{(n)}G_q \VA\subset G_{p+q}\VA,\quad \text{for }a\in G_p \VA,\ n\in
 \Z,
\label{eq:standard-filtration2}
\\
&a_{(n)}G_q \VA\subset G_{p+q-1}\VA,\quad \text{for }a\in G_p \VA,\ n\in
 \Z_{\geq 0}.
\label{eq:standard-filtration3}
\end{align}
It follows that
  $\gr^G \VA=\bigoplus\limits_{p\in \frac{1}{r_0}\Z_{\geq 0}}G_p \VA/G_{p-1}\VA$
is naturally a vertex Poisson algebra,
where we have set $G_p \VA=0$ for $p<0$.

By \cite[Thoerem 4.14]{Li04}
the standard filtration $\{G_p \VA\}$ is characterized 
as the finest increasing filtration of $V$ 
satisfying
(\ref{eq:standard-filtration1})-(\ref{eq:standard-filtration3})
such that
\begin{align}
 \VA_{\Delta}\subset G_{\Delta}\VA.
\label{eq:characterization-of-standard-filtration}
\end{align}
In particular, it is independent of the choice of a 
set of strong generators of $V$.

Both filtrations $\{F^p \VA\}$ and $\{G_p\VA\}$ are 
stable under the action of the Hamiltonian.
Let $F^p\VA_{\Delta}=\VA_{\Delta}\cap F^p \VA$,
$G_p \VA_{\Delta}=\VA_{\Delta}\cap G_p \VA$,
so that $F^p V=\bigoplus_{\Delta} F^pV_{\Delta}$,
$G_p V=\bigoplus_{\Delta}G_p V_{\Delta}$.

\begin{Pro}\label{Pro:F=G}
We have  \begin{align*}
F^p \VA_{\Delta}=G_{\Delta-p}\VA_{\Delta}
	 \end{align*}
for all $p$ and $\Delta$.
Moreover,
the linear isomorphism
\begin{align*}
\gr^F \VA\isomap\gr^G \VA
\end{align*}
is an isomorphism of
 vertex Poisson algebras.
\end{Pro}
\begin{proof}
The second assertion is easily seen from the first.
So let us prove the first assertion.
First, 
we have
by (\ref{eq:characterization-of-standard-filtration})
$\VA_{\Delta}=G_{\Delta}\VA_{\Delta}$,
namely,
\begin{align*}
F^0 V_{\Delta}=G_{\Delta} V_{\Delta}.
\end{align*}
Next we show 
the inclusion 
$F^{p}\VA_{\Delta}\subset G_{\Delta-p}\VA_{\Delta}$
by induction on $p\geq 0$.
Let
 $p>0$.
By 
(\ref{eq:2009:04:02:10:39}),
 $F^p \VA_{\Delta}$
is spanned by the elements
\begin{align}
 a_{(-i-1)}b,
\quad \text{with }a\in \VA_{\Delta_a},\ 
b\in F^{p-i}\VA_{\Delta_b},\
i\geq 1, \
\Delta_a+\Delta_b+i=\Delta.
\label{eq:F=G-1}
\end{align}
By 
the induction hypothesis
$F^{p-i}\VA_{\Delta_b}\subset G_{\Delta_b-p+i}\VA_{\Delta_b}$
for $i\geq 1$.
Because $a\in V_{\Delta_a}\subset G_{\Delta_a}V$,
for the vector $a_{(-i-1)}b$ of the  form
(\ref{eq:F=G-1})
we have
\begin{align*}
 a_{(-i-1)}b\in a_{(-i-1)}G_{\Delta_b-p+i}
\VA_{\Delta_b}
\subset G_{\Delta_a+\Delta_b-p+i}\VA_{\Delta}
=G_{\Delta-p}\VA_{\Delta}.
\end{align*}
Hence
$F^p \VA_{\Delta}\subset G_{\Delta-p}\VA_{\Delta}$.

It remains to show the opposite inclusion 
$G_p \VA_{\Delta}\subset F^{\Delta-p}\VA_{\Delta}$.
We prove 
that any element $v$ 
of $G_p \VA_{\Delta}$
of the form (\ref{eq:2009-04-26-04-26})
belongs to
$F^{\Delta-p}V_{\Delta}$
by induction on $r\geq 0$.
For $r=0$ this is obvious.
So let $r>0$.
Then 
$v=a^{i_1}_{(-n_1)}w$
with
$w=a^{i_2}_{(-n_2)}\dots a^{i_{r}}_{(-n_{r})}\1$,
$n_i\geq 1$,
$\sum \Delta_{a^i}\leq p$,
$\Delta_{a^{i_1}}+\Delta_w+n-1=\Delta$,
where each $a^i$ is homogeneous.
Because $w\in G_{p-\Delta_{a^{i_1}}}V$,
the induction hypothesis
gives that
$w\in F^{\Delta_{a^{i_1}}+\Delta_w-p}V$.
Hence
$a^{i_1}_{(-n)}w\in F^{\Delta_{a^1}+\Delta_w-p+n-1}\VA_{\Delta}
=F^{\Delta-p}\VA_{\Delta}$.
This completes the proof.
\end{proof}

From Proposition \ref{Pro:F=G}
we get the following well-known fact \cite{GabNei03}.
\begin{Co}\label{Co:finitely-strongly-genereted}
Let $\{a^i; i\in I\}$ be  a set of homogeneous vectors 
of $V$.
Then the following are equivalent:
\begin{enumerate}
 \item $\{a^i; i\in I\}$ strongly generates $V$.
\item $\{\bar a^i; i\in I\}$ generates $\Ring{\VA}$.
\end{enumerate}
In particular
$\VA$ is finitely strongly generated
if and only if
 $\Ring{\VA}$ is finitely generated.
\end{Co}
\begin{proof}
Suppose that  $\{a^i; i\in I\}$ 
strongly generates  $V$.
By Proposition \ref{Pro:F=G}
$C_2(V)=F^1 V$ is spanned by the 
vectors of the form (\ref{eq:2009-04-26-04-26})
with $n_1+\dots +n_r-r\geq 1$,
proving that
$\{\bar a^i; i\in I\}$
 generates 
$\Ring{\VA}$.

Conversely, 
suppose that  $\{\bar a^i; i\in I\}$ 
generates $\Ring{V}$.
Then
by
 (\ref{eq:grV-is-a-quotient})
 $\{\bar a^i\}$ 
generates $\gr^F V$ as a differential algebra. 
As  the principal symbol map gives the 
isomorphism
$
V\isomap \gr^F V
$
of vector spaces,
it follows that
 $\{a^i\}$ strongly generates $V$.
\end{proof}

\subsection{Example: universal affine vertex algebras}
\label{Example:universal-afine}
Let $\fing$ be a simple Lie algebra over $\C$,
$\affg=\fing[t,t\inv]\+ \C K$ the Kac-Moody affinization of $\fing$
with the central element $K$.
For $k\in \C$,
define
\begin{align*}
 V^k(\fing)=U(\affg)\*_{U(\fing[t]\+ \C K)}\C_k,
\end{align*}
where $\C_k$ is the one-dimensional representation of
$\fing[t]\+ \C K$
on which $\fing[t]$ acts trivially and 
$K$ acts as the multiplication by $k$.

There is a unique vertex algebra structure in $V^k(\fing)$
such that $\1=1\* 1$ is the vacuum vector
and
\begin{align*}
 Y(x_{(-1)}\1,z)=x(z):=\sum_{n\in \Z}x_{(n)}z^{-n-1}\quad\text{for }x\in \fing,
\end{align*}
where $x_{(n)}=x\* t^n$.
The vertex algebra $V^k(\fing)$ is called
 the {\em universal affine vertex algebra
associated with $\fing$ at level $k$}
(see  \cite[4.7]{Kac98}, \cite[2.4.2]{FreBen04}).

The vertex algebra $V^k(\fing)$ is conformal
by the Sugawara construction, provided that
$k\ne -h\che$,
where $h\che$ is the dual Coxeter number of $\fing$.

\begin{Pro}\label{Pro:in-the-case-of-affine-vertex-algebra}
For any $k\in \C$ we have the following.

\begin{enumerate}
 \item$\Ring{V^k(\fing)}\cong \C[\fing^*]$ as Poisson algebras,
where $\fing^*$ is equipped with the Kirillov-Kostant Poisson structure.
 \item  $\gr^F V^k(\fing)\cong \C[\fing^*_{\infty}]$ as 
vertex Poisson algebras,
where 
$ \C[\fing^*_{\infty}]$ is equipped with the level
$0$ 
 vertex Poisson structure. 
\end{enumerate}
In particular $\Ring{V^k(\fing)}$ and $\gr^F V^k(\fing)$ are
 independent of $k\in \C$.
\end{Pro}
\begin{proof}
Although Proposition  is 
well-known (see e.g., \cite{DonLiMas02,De-Kac05,De-Kac06}),
we include the proof for completeness.

The vertex algebra $V^k(\fing)$ is naturally
$\Z_{\geq 0}$-graded (see e.g., \cite[Example 4.9b]{Kac98}).
We consider the corresponding standard filtration $\{G_p
 V^k(\fing)\}$.

We have  
$V^k(\fing)\cong U(\fing[t\inv]t\inv)$
as vectors spaces.
Under this isomorphism
$G_pV^k(\fing)$ gets identified with
$U_p(\fing[t\inv]t\inv)$,
where $\{U_p(\fing[t\inv]t\inv)\}$ is the standard filtration of the
 enveloping algebra $U(\fing[t\inv]t\inv)$ in the Lie theory,
that is,
$U_p(\fing[t\inv]t\inv)$ is
the linear span of products of at most $p$ elements
in $\fing[t\inv]t\inv$.
It follows that  $\gr^G V^k(\fing)\cong S(\fing[t\inv]t\inv)$ as commutative rings,
and thus $\gr^FV^k(\fing)\cong S(\fing[t\inv]t\inv)$ as 
commutative rings by
Proposition \ref{Pro:F=G}.

Now 
 $S(\fing[t\inv]t\inv)$ is naturally 
isomorphic to $S(\fing)_{\infty}
=\C[\fing^{*}_{\infty}]$ as commutative rings,
and 
we get the  isomorphism
\begin{align*}
\Phi:\C[\fing^*_{\infty}]
\isomap \gr^F V^k(\fing),
\end{align*} 
which is easily seen to be 
 an isomorphism of differential algebras.
It remains  to prove  that
$\Phi$
is a homomorphism of vertex Poisson algebras.
For this, it is 
sufficient to check that
$\Phi(x_{(n)})=\Phi(x)_{(n)}$ only for $x\in \fing$
by \cite[Lemma 3.3]{Li04}.
For $n>0$ this follows immediately from the definition,
and for $n=0$ this is equivalent to  (i) of Proposition,
which is easy to see.
\end{proof}

\section{Associated varieties of vertex algebras and their modules}
\label{section:associated-variety}
For the rest of the paper we will assume that
$V$ to be  finitely strongly generated.
Thus, in particular,
$\Ring{\VA}$ is finitely generated
(Corollary \ref{Co:finitely-strongly-genereted}).

\subsection{Filtration of
  $\VA$-modules}\label{subsection:filtration-of-modules}
Let $M$ be a $\VA$-module.
A {\em compatible filtration}
$\{\Gamma^p M\}$
of $M$ 
is a decreasing filtration
$ M=\Gamma^0M\supset \Gamma^1M\supset \cdots$
satisfying
\begin{align}
&
 a_{(n)}\Gamma^q M\subset \Gamma^{p+q-n-1}M,\quad \text{for }
a\in F^p \VA,\ n\in \Z,
\nno \\
&
 a_{(n)}\Gamma^qM\subset  \Gamma^{p+q-n}M\quad \text{for }
 a\in
F^p \VA,
\ n\geq 0,
\label{eq:Li-fitration2-module}\\
& H \Gamma^q M\subset \Gamma^q M\nno,
\end{align}
where $\{F^p \VA\}$ is the Li filtration of $\VA$.
The associated graded space
$\gr^\Gamma M=\bigoplus \Gamma^p M/\Gamma^{p+1}M$ is naturally a module over 
the vertex Poisson algebra $\gr^F\VA$.
Here we have set $\Gamma^pM=M$ for $p<0$.
Note that each
 subspace $\Gamma^p M/\Gamma^{p+1} M$
 of $\gr^{\Gamma}M$
 is a submodule over 
$\Ring{\VA}=\VA/C_2(\VA)$.

A compatible  filtration $\{\Gamma^p M\}$ is {\em good}\footnote{The definition
of a ``good'' filtration in 
this note  is different form the one given in
\cite{Li04}.}
 if 
it is separated (i.e. $\bigcap \Gamma^p M=0$)
 and 
$\gr^\Gamma M$ is finitely generated over 
the ring $\gr \VA$.

\begin{Rem}
Let $\{\Gamma^p M\}$ be a compatible filtration 
of a $V$-module $M$.
Then 
$\bigcap_p \Gamma^p M$ is a submodule of $M$.
Hence, if $M$ is simple,
we have  either $\bigcap \Gamma^p M=0$ or 
$\bigcap\Gamma^p M=M$.
\end{Rem}

\begin{Lem}\label{Lem:confomal-vs-(\VA,T)}
 Let $\VA$ be a conformal vertex algebra,
$M$ a $\VA$-module,
$\{\Gamma^p M\}$ a compatible filtration.  
Then 
$\gr^\Gamma M$  is a differential $\gr^F V$-module.
\end{Lem}
\begin{proof}
 Let $\omega$ be the Virasoro element of $\VA$.
Then the derivation $ \Gamma^p M/\Gamma^{p+1}\ra \Gamma^{p-1}/\Gamma^p M$,
$\sigma_p(m)\ra \sigma_{p-1}(\omega_{(0)}m)$,
defines 
an action of $T$ with 
 the desired property.
\end{proof}

Let $F^p M$ be a subspace of $M$
spanned by the vectors
\begin{align*}
 a^1_{(-n_1-1)}\dots a^r_{(-n_r-1)}m
\end{align*}
with $a^i\in \VA$,
$m\in M$,
$n_i\in \Z_{\geq 0}$,
$n_1+\dots+n_r\geq p$.
Then 
$\{F^p M\}$ is a compatible  filtration \cite{Li05}.
It is 
separated
because $M$ is lower truncated by our assumption,
see
(the proof of) \cite[Lemma 2.1.4]{Li05}.
The filtration $\{F^p M\}$
is  called the {\em Li filtration} of 
a $V$-module $M$.
Note that 
\begin{align}
F^1M =C_2(M).
\end{align}
The $\gr V$-module structure
of $\gr^F M$
gives a
$R_V$-module structure on
 $M/C_2(M)$:
\begin{align*}
 \bar a. \bar m= \overline{a_{(-1)}m},\quad
\text{for }a\in V,\ m\in M,
\end{align*}
where 
$\bar a=a+C_2(V)$,
$\bar m=m+C_2(M)$.
Note that   $M/C_2(M)$
is also a module over 
$\Ring{V}$ 
viewed as a Lie algebra
by the assignment $ \bar a \mapsto L_{\bar a}$,
where
\begin{align*}
 L_{\bar a} \bar m= \overline{a_{(0)}m},\quad
\text{for }a\in V,\ m\in M.
\end{align*}
These two actions are compatible in the sense that
\begin{align}
L_{ \bar a}\bar b. \bar m= 
{\{\bar a,\bar b\}}.\bar m+
\bar b. L_{\bar a}\bar m.
\label{eq:Poisson-modules}
\end{align}

\begin{Lem}[{\cite[Lemma 4.2]{Li05}}]\label{Lem:Li}
For a  $V$-module $M$,
the $\gr^F \VA$-module $\gr ^F M$ is generated by the 
subspace $M/C_2(M)=F^0 M/F^1 M$.
\end{Lem}

\begin{Lem}\label{Lem:C2-space}
Let 
$\{a^i; i\in I\}$ be a set of strong generators of  $V$.
Then
\begin{align*}
C_2(M)=\haru_{\C}\{a^i_{(-n)}m; i\in I,\ n\geq 2,\
m\in M \}.
\end{align*}
 \end{Lem}
\begin{proof}
Let 
$\{F^p C_2(V)\}$ be the induced filtration.
Then  we have
$\gr^F C_2(V)=\bigoplus_{p\geq 1}F^p M/F^{p+1}M$
and 
\begin{align*}
\gr^F C_2(M)=\haru_{\C}\{\bar a^i_{(-n)}\bar m;i\in I,n\geq 2,
m\in M\}
\end{align*}
by 
(\ref{eq:grV-is-a-quotient}),
Corollary \ref{Co:finitely-strongly-genereted},
and Lemma \ref{Lem:Li},
where $\bar m$ is the image of $m\in M$
in $\gr^F M$.
This proves the assertion.
\end{proof}

We call $M$ {\em finitely strongly generated over $V$}
if $M/C_2(M)$ is finitely generated over $\Ring{\VA}$.

\begin{Lem}
A $V$-module  $M$ is finitely strongly generated if and only if 
there exists a good filtration of $M$. 
\end{Lem}
\begin{proof}
Suppose that 
$M/C_2(M)$ is finitely generated over $\Ring{\VA}$.
Then
the Li filtration of $M$ is good by Lemma \ref{Lem:Li}.

Let us show the opposite direction.
 Suppose  that there exists a good filtration
$\{\Gamma^p M\}$ of $M$, so that
 $\gr^\Gamma M$ is  finitely generated over $\gr^F V$.
Let $v_1,\dots, v_r$
elements of $M$
such that there images
$\bar v_1,\dots,\bar v_r$
generate  $\gr^\Gamma M$ as a $\gr^F V$-module.
Then it follows that $M$ is spanned by the vectors
of the form
$
 (T^{j_1}a_1)\dots (T^{j_r}a_r)  v_i
$
with $a_s\in V$, $j_s\in \Z_{\geq 0}$,
$s=1,\dots,r$.
This means that 
the images of $v_i$'s generate $\gr^F M$ as well,
and the assertion follows.
\end{proof}

\begin{Lem}\label{Lem:C1-cofinite-implies-fg}
Suppose that 
$V$ is conformal 
and $V_0=\C \1$.
Then
for a $V$-module $M$,
$M$ is finitely strongly generated over $V$
if and only if $M$ is $C_1$-cofinite.
\end{Lem}
\begin{proof}
Observe that the grading on $V$ induces a grading on $\Ring{\VA}$
such that
\begin{align}
 \Ring{\VA}=\bigoplus_{\Delta\in \frac{1}{r_0}\Z_{\geq 0}}(\Ring{\VA})_{\Delta},
\quad
(\Ring{\VA})_0=\C.
\label{eq:conical}
\end{align}

The inclusion $C_2(M)\hookrightarrow C_1(M)$
gives the surjection 
\begin{align*}
\eta: M/C_2(M)\twoheadrightarrow M/C_1(M).
\end{align*}
As easily seen $\eta$ is a homomorphism of 
$\Ring{\VA}$-modules, where
$M/C_1(M)$ is considered as a trivial $\Ring{\VA}$-module,
and we have
$\ker \eta= \Ring{\VA}^* \left( M/C_2(M)\right)$,
where
 $\Ring{\VA}^*$ is the argumentation ideal of $\Ring{\VA}$:
\begin{align}
 \Ring{\VA}^*=\bigoplus_{\Delta>0}(\Ring{\VA})_{\Delta}.
\end{align}
\end{proof}

\subsection{Associated varieties of  vertex algebras
and their modules}
\begin{Lem}\label{Lem:annihilator-is-Poisson}
Let $M$ be a $V$-module.
 \begin{enumerate}
\item The annihilator $\Ann_{\Ring{\VA}}(M/C_2(M))$
of $M/C_2(M)$ in $\Ring{\VA}$ is a Poisson ideal of $\Ring{\VA}$.
  \item Let
$\{\Gamma^p M\}$ be  a compatible filtration.
Then 
$\Ann_{(\Ring{\VA})_{\infty}}(\gr^{\Gamma }M)$
is a 
$(\Ring{\VA})_{\infty}$-submodule of $(\Ring{\VA})_{\infty}$.
If  $V$ is conformal, then 
$\Ann_{(\Ring{\VA})_{\infty}}
(\gr^{\Gamma }M)$ is a vertex Poisson algebra ideal of 
$(\Ring{\VA})_{\infty}$.
 \end{enumerate}
\end{Lem}
\begin{proof}
 (i) follows from
\eqref{eq:Poisson-modules}.
(ii)
By (\ref{eq:axiom4-for-VPA}),
$\Ann_{\gr^F V}(\gr^{\Gamma}M)$ is a submodule
of $\gr^F V$.
Because  $(\Ring{\VA})_{\infty}$ acts on $\gr^{\Gamma}M$
via the surjective homomorphism 
$\Phi:(\Ring{\VA})_{\infty}\ra \gr^F V$ of vertex Poisson algebras,
 $\Ann_{(R_V)_{\infty}}(\gr^{\Gamma} M)
=\Phi^{-1}(\Ann_{\gr^F
 V}(\gr^{\Gamma}M))$
is a submodule of $(R_V)_{\infty}$.
If $V$ is conformal then  $\gr^{\Gamma}M$ is a differential 
$\gr^F V$-module by Lemma \ref{Lem:confomal-vs-(\VA,T)}.
It follows that
$\Ann_{\gr^F V}(\gr^{\Gamma}M)$ is 
an ideal and hence so is
  $\Ann_{(R_V)_{\infty}}(\gr^{\Gamma} M)$.
\end{proof}

 Define the {\em associated variety} $\Ass{\VA}$ of 
$\VA$
by
\begin{align*}
 \Ass{\VA}=\Spec \Ring{\VA}.
\end{align*}
More generally, for a finitely strongly generated  $\VA$-module $M$,
 define  the associated  variety  $\Ass{M}$ 
of  $M$ by 
\begin{align*}
 \Ass{M}
&=\on{supp}_{\Ring{\VA}}(M/C_2(M))\\
&=\{\p\in \Spec \Ring{\VA};
\p\supset \Ann_{\Ring{\VA}}(M/C_2(M))\}.
\end{align*}
 By Lemma \ref{Lem:annihilator-is-Poisson},
$\Ass{M}$ is a Poisson subvariety of $\Ass{\VA}$.


The following assertion is clear.
\begin{Lem}\label{Lem:C2-vs-zero-assiciated-variety}
Let $M$ be a finitely strongly generated  $\VA$-module.
Then the following are equivalent:
\begin{enumerate}
 \item $M$ is $C_2$-cofinite.
\item $\dim \Ass{M}=0$.
\end{enumerate}
\end{Lem}
\begin{Rem}\label{Rem:if-graded}
Let $V$ be as in Lemma \ref{Lem:C1-cofinite-implies-fg}.
Then by (\ref{eq:conical}) 
the variety $\Ass{\VA}$ is conical.
If $M$ is a graded $\VA$-module then 
$\Ass{M}$ is also conical.
Hence
\begin{align*}
\text{$M$ is $C_2$-cofinite}\quad \iff \quad \Ass{M}=\{0\}
\quad\text{(as topological spaces).}
\end{align*}
This is equivalent to that,
for any homogeneous element $a\in V$ with  $\Delta_a>0$,
 $a+C_2(V)$ 
acts nilpotently on $ M/C_2(M)$.
\end{Rem} 

\subsection{Singular support of $\VA$-modules}
Let 
$M$ be a finitely strongly generated  $\VA$-module,
$\{\Gamma^p M\}$ a good filtration of $M$.
Define the {\em singular support} $SS(M)$
of $M$
by
\begin{align*}
 SS(M)
&=\on{supp}_{(\Ring{\VA})_{\infty}}(\gr^{\Gamma} M)\\
&=
\{p\in \Spec (\Ring{\VA})_{\infty};
\p \supset \Ann_{(\Ring{\VA})_{\infty}}(\gr^\Gamma M)\}.
\end{align*}
Then  $SS(M)$ is a closed subscheme of the
infinite jet scheme $(\Ass{\VA})_{\infty}=\Spec (\Ring{\VA})_{\infty}$.
It is well-known that
 $SS(M)$ is independent of the choice of 
a  good filtration of $M$.

Let 
\begin{align*}
\pi_m: (\Ass{\VA})_{\infty}\ra (\Ass{\VA})_m
\end{align*}
be 
the natural projection.

\begin{Lem}\label{Lem:C2-variety-vs-singular-support}
Let $M$ be a finitely strongly generated $V$-module.
\begin{enumerate}
 \item We have 
 $\Ass{M}=\pi_0(SS(M))$.
\item If $V$ is conformal then 
$SS(M)\subset (\Ass{M})_{\infty}$.
\end{enumerate}
\end{Lem}
\begin{proof}
We may take 
 the Li filtration $\{F^p M\}$ as a good filtration.
By Lemma \ref{Lem:Li},
\begin{align}
\Ann_{\gr^F \VA}(\gr^F M)
\cap \Ring{\VA}=\Ann_{\Ring{\VA}}
(M/C_2(M))
\label{eq:Ass-vs-SS}
\end{align}Hence $X_M=\pi_0(SS(M))$.
Next assume that $V$ is conformal.
By
Lemma \ref{Lem:annihilator-is-Poisson}
$\Ann_{(\Ring{\VA})_{\infty}}(\gr^F M)$
is  $T$-stable.
Thus from
(\ref{eq:Ass-vs-SS})
it follows that
$\Ann_{\gr^F \VA}(\gr^F M)$
contains the 
defining ideal of $(X_M)_{\infty}$,
that is,
 the
minimal $T$-stable ideal of 
$\C[(X_V)_{\infty}]$  containing
$\Ann_{\Ring{\VA}}(M/C_2(M))$.
This shows that  $SS(M)\subset (X_M)_{\infty}$.
\end{proof}

A $V$-module $M$ 
 is called {\em lisse} 
if $M$ is finitely strongly generated  and $SS(M)$ is zero-dimensional,
or equivalently,
$\pi_m(SS(M))$ is zero-dimensional  for all $m\geq 0$.
A vertex algebra $V$ is called  lisse if it is lisse as
a module over itself.

\begin{Rem}
  Suppose that
$\VA$ is $\Q_{\geq 0}$-graded 
and $\VA_0=\C$.
Then as in Remark \ref{Rem:if-graded}
we see that
 a $C_1$-cofinite  $V$-module
$M$ is lisse if and only if 
$SS(M)=\{0\}$.
or 
equivalently,
for  any homogeneous element $a$ with 
$\Delta_a>0$,
its principal symbol $\sigma_p(a)$
acts nilpotently on  in $\gr^{F} M$.
\end{Rem}
\begin{Th}\label{Th:C2-vs-lisse}
The following are equivalent.
\begin{enumerate}
 \item $\VA$ is $C_2$-cofinite.
\item $\VA$ is lisse.
\end{enumerate}
\end{Th}
\begin{proof}
 The direction (ii) $\Rightarrow$ (i) 
immediately follows from 
 Lemma \ref{Lem:C2-variety-vs-singular-support}.
 The direction (i) $\Rightarrow$ (ii)  follows from the fact that
$SS(V)\subset (\Ass{\VA})_{\infty}$
and the jet scheme of $0$-dimensional variety is $0$-dimensional.
\end{proof}


The same assertion holds for the modules
provided that 
$\gr^{\Gamma} M$ is a differential $\gr V$-module
for a good grading $\{\Gamma^p M\}$.
This is the case when $V$ is conformal:
\begin{Th}\label{Th:C2-vs-lisse-modules}
Suppose that $V$ is conformal 
and $M$ is a finitely strongly generated $V$-module.
Then the following are equivalent.
\begin{enumerate}
 \item $M$ is $C_2$-cofinite,
or equivalently, 
$\dim \Ass{M}=0$.
\item $M$ is lisse.
\end{enumerate}
\end{Th}
\begin{proof}
By Lemma \ref{Lem:confomal-vs-(\VA,T)}
the assertion follows from 
in the same manner as
 Theorem \ref{Th:C2-vs-lisse}.
\end{proof}

%

\subsection{Example: Virasoro vertex algebras}
Let $\mc{L}=\bigoplus_{n\in \Z}\C L_n\+\C \mathbf{c}$ be the Virasoro
algebra, with the commutation relation
\begin{align*}
 &[L_m,L_n]=(m-n)L_{m+n}+\frac{(m^3-m)}{12}\delta_{m+n,0}\mathbf{c},\\
&[\mathbf{c}, \mc{L}]=0.
\end{align*}
Define the subalgebras
$\mc{L}_{\geq 0}=\bigoplus_{n\in \Z}\C L_n \+ \C \mathbf{c}$,
$\mc{L}_{<0}=\bigoplus_{n<0}\C L_n$,
so that $\mc{L}=\mc{L}_{<0}\+ \mc{L}_{\geq 0}$.
For $c\in \C$,
denote by $\C_c$  the one-dimensional representation of $\mc{L}_{\geq 0}$
on which $L_n$ acts trivially and $\mathbf{c}$ acts as the
multiplication
by $c$.
Define
$M_c=U(\mc{L})\*_{U(\mc{L}_{\geq 0})}\C_c$.
Then $L_{-1}(1\*1)\in M_c$ is annihilated by all
$L_n$ with $n>0$.
Set
\begin{align*}
 Vir^c=M_c/U(\mc{L})L_{-1}(1\*1).
\end{align*}
As is well-known
there is a unique vertex algebra 
structure on
$Vir^c$ such that the image of $1\* 1$
is the vacuum vector $\1$ and 
\begin{align*}
 &Y(L_{-2}\1,z)=L(z):=\sum_{n\in \Z}L_{n}z^{-n-2}.
\end{align*}
The vertex algebra  $Vir^c$ is called
 the {\em universal Virasoro vertex algebra with central charge
$c$}.
Any   $\mc{L}$-module 
with central charge $c$
(i.e., $\mathbf{c}$ acts as the multiplication by $c$)
on which $L(z)$ is a field
can be  considered as a $Vir^c$-module.

We have
 \begin{align}
		 \Ring{Vir^c}
\cong \C[x]
\label{eq:R-for-Virasoro}
		\end{align} (with the trivial Poisson structure),
where $x$ is the image of $L_{-2}\1$.
Let
$N_c$ its unique maximal submodule
of $Vir_c$,
and 
$Vir_c=\Vir^c/N_c$ 
its unique  simple quotient.

\begin{Pro}
  The following are equivalent.
\begin{enumerate}
 \item $Vir_c$ is $C_2$-cofinite.
\item $Vir^c$ is reducible.
\item $c= 1-6(p-q)^2/pq$
for some $p,q\in \Z_{\geq 2}$
such that $(p,q)=1$.
(These are precisely the central charges of the minimal series
representations of ${\mc{L}}$.)
\end{enumerate}
\end{Pro}
\begin{proof}
It is known that
the image of $N_c$ in $\Ring{Vir^c}$ is nonzero
if $N_c\ne 0$ (see e.g.\ \cite[Lemma 4.2 and Lemma 4.3]{Wan93} or \cite[Proposition 4.3.2]{GorKac07}).
Therefore 
$\Ass{Vir_c}=\{0\}$ 
if and only if $Vir^c$ is not irreducible.
This happens if and only if 
the central charge is of the form in (iii)
(\cite{Kac74,FeuFuc84,GorKac07}).
\end{proof}

We now compare the $C_2$-cofiniteness condition with 
the zero singular support condition of Beilinson, Feigin and Mazur \cite{BeiFeiMaz}.
Let $\{U_p(\mc{L})\}$ be the standard filtration of $U(\mc{L})$
as in Introduction.
Then the associated graded algebra $\gr U(\mc{L})$
is isomorphic to the symmetric algebra $S(\mc{L})$.
Let $M$ be a highest weight representation of $\mc{L}$
with central charge  $c$,
$v$ the highest weight vector of $M$.
Then $\{U_p(\mc{L})v\}$ defines an
increasing filtration of $M$
compatible with the
standard filtration of $U(\mc{L})$.
(Note that 
$U_p(\mc{L})v=U_p(\mc{L}_{<0})v$.)
Let $\gr^{Lie}M$ denote the associated graded space
$\bigoplus_{p}U_{p}(\mc{L})v
/U_{p+1}(\mc{L})v$.
Then $\gr^{Lie}M$ is an  $S(\mc{L})$-module
generated by the image of $v$.
The singular support $SS^{BFM}(M)$
defined in \cite{BeiFeiMaz}
is by definition the support $\on{supp}_{S(\mc{L})}\gr^{Lie}(M)$.
Clearly, we have $SS^{BFM}(M)\subset \mc{L}_{<0}^*$.
\begin{Pro}\label{Pro:comparison}
 Let $M$ be a highest weight representation of
$\mc{L}$ with central charge $c$.
Then the following are equivalent.
\begin{enumerate}
 \item $M$ is $C_2$-cofinite.
\item $SS^{BFM}(M)=\{0\}$.
\end{enumerate}
\end{Pro}
\begin{proof}
First, by Lemma \ref{Lem:C2-space}
we have
\begin{align*}
 C_2(M)=\haru_{\C}\{ L_{-n}m; m\in M, n\geq 3\}
\end{align*}
because $Vir^c$ is strongly generated by
$L_{-2}\1$.
It follows that
$M/C_2(M)$ is spanned by the images of the vectors
of the form
$
 L_{-2}^m L_{-1}^n v
$ with $m,n\geq 0$,
where 
 $v$ is the highest weight vector of $M$.
In particular
$M/C_2(M)$ is generated by
the image of the vectors 
$L_{-1}^nv$ with $n\geq 0$
over $R_{Vir^c}$
by (\ref{eq:R-for-Virasoro}).

Suppose that
$SS^{BFM}(M)=\{0\}$.
We then have
$L_{-1}^p v\in U_{p-1}(\mc{L}_{<0})v$ 
for a  sufficiently large $p$.
By considering the $L_0$-eigenvalue
we find that this is
equivalent
to $L_{-1}^p v\in \sum_{n\geq 2}L_{-n}U(\mc{L}_{<0})v$.
This happens if and only if
$M/C_2(M)$ is 
finitely generated over $\Ring{Vir^c}$.
We also have 
$L_{-2}^p v\in U_{p-1}(\mc{L}_{<0})v$ 
for a  sufficiently large $p$.
By taking account 
of the $L_0$-eigenvalue
it follows that
this 
is equivalent
to $L_{-2}^p v\in \sum_{n\geq 3}L_{-n}U(\mc{L}_{<0})v
= C_2(M)$.
Therefore  $M$ is $C_2$-cofinite.

Conversely,
 suppose that
$M$ is $C_2$-cofinite.
From the above argument it follows that
$L_{-1}$ and $L_{-2}$ act nilpotently on 
$\gr^{Lie}M$.
Then $SS^{BFM}(M)$ must be $\{0\}$,
because 
$L_{-1}$ and $L_{-2}$ generates 
$\mc{L}_{<0}$ and
$\sqrt{\Ann_{S(\mc{L}_{<0})}\gr^{Lie}M}$ is involtive
by Gabber's theorem.
\end{proof}

By Proposition \ref{Pro:comparison}
a result of
 \cite{BeiFeiMaz} can be read as follows.
\begin{Th}[\cite{BeiFeiMaz}]
Let $M$ be an irreducible highest weight representation of $\mc{L}$
with central charge $c$.
Then the following are equivalent.
\begin{enumerate}
 \item $M$ is $C_2$-cofinite.
\item $M$ is isomorphic to one of the minimal series representations of $\mc{L}$.
\end{enumerate}
\end{Th}

\subsection{Example: affine vertex algebras}
Let $\fing$,
$\affg$,
$V^k(\fing)$ be as in Example \ref{Example:universal-afine}.

For a weight $\lam$ of $\affg$,
let $L(\lam)$ be the 
irreducible (graded) representation of $\affg$
of highest weight $\lam$.
If $\lam$ is of level $k$,
that is,
if $\lam(K)=k$,
then $L(\lam)$ can be naturally considered as a (simple)
module over $V^k(\fing)$.
In particular
$L(k\Lam_0)$
(in the notation of \cite{Kac90})
is isomorphic to
the unique simple graded quotient
of $V^k(\fing)$. 

Let $\lam(K)=k$.
Because 
\begin{align*}
C_2(L(\lam))=\haru_{\C}\{x_{(-n)}m; x\in \fing,\ m\in L(\lam),\
n\geq 2\}
\end{align*} 
by Lemma \ref{Lem:C2-space},
it follows that
$L(\lam)$ 
 is finitely strongly generated over $V^k(\fing)$
if and only if 
the highest weight vector of $L(\lam)$
generates
a finite-dimensional $\fing$-module,
or equivalently,
the restriction of $\lam$ to the Cartan subalgebra of $\fing$
is integral dominant.
This happens if and only if 
$L(\lam)$
is a direct sum of finite-dimensional
$\fing$-modules.
In this case
$X_{L(\lam)}$ is a $\Ad G$-invariant, conic,
 Poisson subvariety of $\fing^*$.

We give a simple proof of the following well-known
assertion (cf. \cite{Zhu96}).
 \begin{Pro}
\label{Pro:integral-modules-are-C2-cofinite}
Let $\lam$ be a weight of $\affg$,
and set $k=\lam(K)$.
Then
$L(\lam)$ is a $C_2$-cofinite  $V^k(\fing)$-module 
if it is an integrable representation of $\affg$.
In particular 
the simple affine vertex algebra  $L(k\Lam_0)$ is 
$C_2$-cofinite if $k\in \Z_{\geq 0}$.
\end{Pro}
\begin{proof}
Suppose that
 $L(\lam)$ is integrable
and set  $J=\Ann_{\Ring{V^k(\fing)}}L(\lam)/C_2(L(\lam))$.
Let
$x_{\alpha}$ is any root vector of $\fing$.
Because
$(x_{\alpha})_{(-1)}$ acts locally nilpotently on $L(\lam)$,
its image $x_{\alpha}\in \C[\fing^*]=\Ring{V^k(\fing)}$
belongs to  $\sqrt{J}$.
As
$\sqrt{J}$ is an Poisson ideal of $\C[\fing^*]$
by 
 Corollary  \ref{Co:radical-of-VPA-ideal}
and Lemma \ref{Lem:annihilator-is-Poisson},
$\sqrt{J}$ contains $\{\fing,x_{\alpha}\}=\fing$,
proving that
$\Ass{L(\lam)}=\{0\}$.
\end{proof}
One can show that
the converse of Proposition \ref{Pro:integral-modules-are-C2-cofinite}
is also true, that is,
$L(\lam)$ is $C_2$-cofinite if and only if it is integrable,
see \cite{DonMas06,Ara09b}.

\bibliographystyle{jalpha}
\bibliography{math}

\end{document}